\documentclass[12pt]{amsart}

\usepackage[a4paper,margin=2.5 cm]{geometry}
\address{\noindent
Pierre Youssef,\newline
University of Alberta\newline
Department of Mathematical and Statistical Sciences\newline
Edmonton, AB, Canada.\newline
\texttt{e-mail: \small pyoussef@ualberta.ca}
}

\date{}
\usepackage{amsthm,hyperref,lmodern,mathrsfs}

\usepackage[latin1]{inputenc}
\usepackage[T1]{fontenc}
\usepackage{graphicx}
\usepackage[frenchb,english]{babel}
\usepackage{calc}
\usepackage{amsmath,amsthm}
\usepackage{graphicx}
\usepackage{url}
\usepackage{times, amssymb, amscd, mathrsfs, graphicx, color}
\usepackage{enumerate}

\newtheorem{theo}{Theorem}[section]

\newtheorem{coro}[theo]{Corollary}
\newtheorem{lem}[theo]{Lemma}

\newtheorem{Rq}[theo]{Remark}

\title[]{Extracting a basis with fixed block inside a matrix}
\author{Pierre Youssef}
\begin{document}

\maketitle

\begin{abstract}
Given $U$ an $n\times m$ matrix of rank $n$ whose columns are denoted by $(u_j)_{j\leqslant m}$, several authors have
already considered the problem of finding a subset $\sigma\subset\{1,\ldots,m\}$ such that $(u_i)_{i\in\sigma}$ span $\mathbb{R}^n$ and $\sqrt{{\rm Tr}\left(\left(\sum_{i\in\sigma} u_iu_i^t\right)^{-1}\right)}$ is minimized. 
In this paper, we generalize this problem by selecting arbitrary rank matrices instead of rank one matrices. 
Another generalization is considering the same problem while allowing a part of the matrix to be fixed. 
The methods of selection employed develop into algorithms.
\end{abstract}

\section{Introduction}

Let $U$ be an $n\times m$ matrix, we see $U$ as an operator from 
$l_2^m$ to $l_2^n$. We denote $\Vert U\Vert$ the operator norm of $U$ while 
the Hilbert-Schmidt 
norm of $U$ is given by $\Vert U\Vert_{\rm HS}= \sqrt{{\rm Tr}\left(UU^t\right)}$.  
The stable rank of $U$ is given by ${\rm srank}(U):= \Vert U\Vert_{\rm HS}^2/\Vert U\Vert^2$. 
Note that the stable rank is always less or equal to the rank. 
We denote by $s_{\max}$ and $s_{\min}$ the largest and smallest singular value, respectively. 
Given $\sigma\subset \{1,...,m\}$, we denote $U_{\sigma}$ the restriction 
of $U$ to the columns with indices in $\sigma$ i.e. $U_{\sigma}=UP_{\sigma}^t$ 
where $P_{\sigma}: \mathbb{R}^m\longrightarrow \mathbb{R}^{\sigma}$ is the canonical 
coordinate projection. When $\sigma=\emptyset$, we have $U_{\sigma}=0$. 
Finally, if $A$ and $B$ are $n\times n$ symmetric matrices, the notation $A\preceq B$ means that 
$B-A$ is positive semidefinite.

Column subset selection usualy refers to extracting from a matrix a column submatrix that has some distinguished properties. 
First results on column selection problems were obtained by Kashin \cite{MR604848}. The aim of the selection was to find 
a submatrix which minimizes the operator norm among all restrictions of the same size. It was later sharpened 
in \cite{MR1001700} and \cite{kashin-tzafriri} where, for any $n\times m$ matrix $U$ and any $\lambda \leqslant 1/4$, 
it is proved that there exists $\sigma\subset\{1,\ldots,m\}$ of size $\lambda m$ with $\Vert U_{\sigma}\Vert 
\leqslant C(\sqrt{\lambda} \Vert U\Vert + \Vert U\Vert_{\rm HS}/\sqrt{m})$, with $C$ being a universal constant. 
In \cite{MR890420}, Bourgain-Tzafriri considered selecting a block of columns which is well invertible i.e. 
whose smallest singular value is bounded away from zero; their result states that there exist universal constants 
$C$ and $C'$ such that for any $n\times n$ matrix $A$ whose columns are of norm $1$, one can find $\sigma\subset\{1,\ldots,m\}$ 
of size at least $C{\rm srank}(A)$ such that $s_{\min}(A_{\sigma}) \geqslant C'$. 
Later, Vershynin \cite{MR1826503} extended the restricted invertibility principle of Bourgain-Tzafriri to the case 
of rectangular matrices. Moreover, he also studied the extraction of a well conditioned submatrix and proved that for any $\varepsilon\in (0,1)$ and 
any $n\times m$ matrix $U$, there exists $\sigma\subset \{1,\ldots,m\}$ of size at least  $(1-\varepsilon){\rm srank}(U)$ such that 
$s_{\max}(U_{\sigma})/s_{\min}(U_\sigma)\leqslant \varepsilon^{c \log (\varepsilon)}$. 
These results were very important in geometric functional analysis and had several applications. However, the proofs were not constructive 
as they were based on random selection and made use of Grothendieck's factorization theorem.  In \cite{MR2807539}, Tropp was able 
to provide a randomized polynomial time algorithm to achieve the selection promised by the results of Bourgain-Tzafriri and Kashin-Tzafriri. 
In \cite{MR2956233}, Spielman-Srivastava produced a deterministic  polynomial time algorithm to find a well invertible submatrix inside any rectangular matrix, 
generalizing and improving the restricted invertibility principle of Bourgain-Tzafriri. Their proof is inspired by the method developed by Batson-Spielman-Srivastava 
\cite{MR2780071} to find a spectral sparsifier of a graph. In that paper, they gave a deterministic polynomial time algorithm 
to find a submatrix, with much fewer columns, which approximate the singular values of the original rectangular matrix. 
More precisely, for any $\varepsilon\in (0,1)$ and any $n\times m$ matrix $U$, they showed the existence of $\sigma\subset \{1,\ldots,m\}$ 
of size $O(n/\varepsilon^2)$ such that $(1-\varepsilon)UU^t\preceq U_\sigma U_\sigma^t\preceq (1+\varepsilon)UU^t$.
The main idea is to select the columns one by one, study the evolution of the singular values and keep controlling this 
evolution until the extraction is done. This idea was exploited in \cite{Y1} to give a deterministic polynomial time algorithm 
which finds the submatrix promised by the result of Kashin-Tzafriri. Similar tools were developed in \cite{Y2} in order 
to extract a well conditioned submatrix improving the result obtained by Vershynin \cite{MR1826503}. More precisely, we proved that for any $\varepsilon\in (0,1)$ 
and any $n\times m$ matrix $U$ with columns of norm $1$, there exists $\sigma\subset\{1,\ldots, m\}$ of size at least $(1-\varepsilon)^2{\rm srank}(U)$ 
such that the singular values of $U_\sigma$ lie between $\varepsilon/(2-\varepsilon)$ and $(2-\varepsilon)/\varepsilon$. 
When $\varepsilon$ is close to $1$, this result can be seen as dual to the one in \cite{MR2780071} and 
allowed us to produce a deterministic polynomial time algorithm to partition any $n\times n$ zero diagonal 
matrix into $\log(n)$ square blocks of small norm around the diagonal. Doing such a partition with a number of blocks independent 
of the dimension is known as the paving problem \cite{MR525951} 
which is equivalent to the Kadison-Singer problem , recently solved by Marcus-Spielman-Srivastava \cite{MSS}.

In this paper, we are interested in another column selection problem. Let $U$ be 
an $n\times m$ matrix of rank $n$, we want to extract an $n\times n$ invertible matrix 
inside $U$ which minimizes the Hilbert-Schmidt norm of the inverse. The invertibility question 
is closely related to the restricted invertibility studied in \cite{MR890420} and \cite{MR2956233}; 
the difference here is that we want to extract exactly $rank(U)$ columns while in the restricted 
invertibility principle, one is only allowed to extract a number of columns strictly less than 
$srank(U)$. 
We learned about this problem in a paper of Gluskin-Olevskii \cite{MR2852316},  who proved the following:

\begin{theo}[Gluskin-Olevskii]\label{gluskin-olevskii}
Let $U$ be an $n\times m$ matrix of rank $n$. Then there exists $\sigma\subset \{1,..,m\}$ of size $n$ 
such that $U_{\sigma}$ is invertible and 
$$
\Vert U_{\sigma}^{-1}\Vert_{\rm HS}^2\leqslant (m-n+1)\cdot {\rm Tr}\left(UU^t\right)^{-1}.
$$
\end{theo}

The proof of Gluskin-Olevskii is only existential. At first, we were interested in giving a deterministic 
algorithm achieving the extraction. This can be done by carefully removing the "bad" columns of $U$. 
However, after doing so, it turned out that similar algorithms were already established in \cite{MR2305124} and later in 
\cite{MR3121759}\footnote{ We would like to thank Nick Harvey for pointing out to us this reference.}. We refer to \cite{MR3121759} for 
a complete exposition of this problem. We include 
our algorithmic proof in the Appendix as it would be simpler for the reader to understand the proof of the main result which is in the same spirit. 

Our aim in this paper is to generalize Theorem~\ref{gluskin-olevskii} in different ways. The first generalization 
is considering the same problem with the constraint of keeping inside the extraction a block chosen at the beginning. 
This can be useful if we already know the existence of a "good" block inside $U$. In that case, we want to keep this block and 
try to complement it instead of doing the full extraction. This has the advantage of reducing the complexity of the algorithm 
and improving the estimate on the Hilbert-Schmidt norm of the inverse of the restriction. 
In \cite{MR3121759}, an interesting connection between low-stretch
spanning trees and subset selection was made. It is proved there, how finding 
a low-stretch tree inside a graph is reduced to the problem studied in Theorem~\ref{gluskin-olevskii} for 
an appropriately chosen matrix $U$. Therefore, selecting columns inside the matrix $U$ corresponds to selecting edges of the graph. 
In our generalization, doing the extraction while keeping a fixed block of columns would reflect in 
finding a low-stretch tree under the constraint of keeping a subtree chosen at the beginning. 
We refer the interested reader to section~5 in \cite{MR3121759} for more information on this subject. 
Let us note that the idea of selecting edges of a graph while required to keep a subgraph has already been exploited 
in \cite{MR2743254}. In that paper, the authors were interested in finding a spectral sparsifier similarly to what is done in \cite{MR2780071} 
with the constraint of preserving a subgraph.

The second generalization concerns arbitrary rank matrices; indeed, Theorem~\ref{gluskin-olevskii} states 
that given $UU^t= \sum_{i\leqslant m} u_iu_i^t$, there exists $\sigma\subset \{1,..,m\}$ of size $n$ 
such that $\sum_{i\in \sigma} u_iu_i^t$ is invertible and 
$$
{\rm Tr}\left(\sum_{i\in \sigma} u_iu_i^t\right)^{-1}\leqslant (m-n+1)\cdot {\rm Tr}\left(UU^t\right)^{-1}.
$$
Whereas Theorem~\ref{gluskin-olevskii} only deals with rank one matrices, we consider the same problem 
when replacing the rank one matrices by arbitrary rank matrices. Extending column selection problems 
to an arbitrary rank setting is natural and can be very useful. In \cite{CHS}, the authors extended the spectral sparsification 
of \cite{MR2780071} to an arbitrary rank setting and showed how it can be used in order to find a spectral sparsifier of a graph 
while controlling additional parameters like the cost functions. 

The main result of this paper is the following:

\begin{theo}\label{general-result}
Let $A, B, B_1,..., B_m$ be $n\times n$ positive semidefinite matrices such that 
$A=B+ \sum_{i\leqslant m} B_i$ is of rank $n$. Then for any 
$k\geqslant n - \left\lfloor {\rm Tr}\left(A^{-1}B\right)\right\rfloor$, 
there exists $\sigma_k$ of size $k$ such that 
$$
{\rm Tr}\left(A_{\sigma_k}^{-1}\right)\leqslant {\rm Tr}\left(A^{-1}\right)\cdot\left[\frac{m-n
+{\rm Tr}\left(A^{-1}B\right)+1}{k-n+1+{\rm Tr}\left(A^{-1}B\right)}\right] 
- \frac{(m-k)\cdot{\rm Tr}\left(A^{-2}B\right)}{k-n+1+{\rm Tr}\left(A^{-1}B\right)},
$$
where $A_{\sigma_k} = B+\sum_{i\in\sigma_k}B_i$. 

In particular, 
there exists $\sigma \subset \{1,...,m\}$ with 
$$
\vert \sigma\vert = n- \left\lfloor{\rm Tr}\left( A^{-1}B\right)\right\rfloor
$$ such that 
$$
{\rm Tr}\left(A_{\sigma}^{-1}\right) \leqslant \left(m-n+1+{\rm Tr}\left(A^{-1}B\right)\right)\cdot{\rm Tr}\left(A^{-1}\right) 
-{\rm Tr}\left(A^{-2}B\right)\left(m-n+{\rm Tr}\left(A^{-1}B\right)\right),
$$
where $A_\sigma =B+ \sum_{i\in \sigma} B_i$.
\end{theo}

With the previous notations, $B$ is the matrix that we would like to preserve while selecting among the $(B_i)_{i\leqslant m}$ 
the ones that will minimize the trace of the inverse of the corresponding sum. 

To have a better understanding on the limitation on $k$ in the previous statement, take $A$ to be the identity on 
$\mathbb{R}^n$, $B$ a projection of rank $l$ and $B_1,\ldots, B_m$ rank one projections. Note that in this case $\left\lfloor {\rm Tr}\left(A^{-1}B\right)\right\rfloor=l$. 
To form a rank $n$ matrix while keeping $B$, one should at least choose $n-l$ rank one matrices which is exactly the limitation given by $n - \left\lfloor {\rm Tr}\left(A^{-1}B\right)\right\rfloor$.
\vskip 0.5cm 

In section 2, we will discuss how this result generalizes Theorem~\ref{gluskin-olevskii} and state 
direct consequences of it. The proof of Theorem~\ref{general-result} will be given in section 3.

\section{Derived results}

 Let us first derive the rank one case of Theorem~\ref{general-result}. This case is interesting 
since the selection can be traced as a selection of columns inside the matrix. 
Preserving the matrix $B$ in Theorem~\ref{general-result} plays the role of preserving 
a block inside the matrix. Precisely, we have the following:

 \begin{theo}\label{general-result-rank-one}
Let $U$ be an $n\times m$ matrix of rank $n$ and $\nu\subset\{1,..,m\}$. 
Denote by $V$ the block of columns inside $U$ 
corresponding to $\nu$ i.e. $V= U_\nu$. Let $A=UU^t$ and $B= VV^t$. 
For any $k\geqslant n -\left\lfloor{\rm Tr}\left(A^{-1}B\right)\right\rfloor$, 
there exists $\sigma_k'\subset \nu^c$ of size $k$ such that if $\sigma_k=\sigma_k' \cup \nu$ 
then $U_{\sigma_k}$ is of rank $n$ and
$$
{\rm Tr}\left(U_{\sigma_k}U_{\sigma_k}^t\right)^{-1}\leqslant {\rm Tr}\left(A^{-1}\right)\cdot\left[\frac{\vert\nu^c\vert-n
+{\rm Tr}\left(A^{-1}B\right)+1}{k-n+1+{\rm Tr}\left(A^{-1}B\right)}\right] 
- \frac{\left(\vert\nu^c\vert-k\right)\cdot{\rm Tr}\left(A^{-2}B\right)}{k-n+1+{\rm Tr}\left(A^{-1}B\right)}.
$$
In particular, there exists $\sigma'\subset \nu^c$ with 
$$
\vert \sigma'\vert =  n- \left\lfloor{\rm Tr}\left(A^{-1}B\right)\right\rfloor,
$$
such that
$$
{\rm Tr}\left(U_{\sigma}U_{\sigma}^t\right)^{-1}\leqslant \left(\vert\nu^c\vert-n+1+{\rm Tr}\left(A^{-1}B\right)\right)\cdot{\rm Tr}\left(A^{-1}\right) 
-{\rm Tr}\left(A^{-2}B\right)\left(\vert\nu^c\vert-n+{\rm Tr}\left(A^{-1}B\right)\right),
$$
where $\sigma=\sigma' \cup \nu$.
\end{theo}

\begin{proof}
Denote by $(u_i)_{i\leqslant m}$ the columns of $U$. Clearly, $UU^t= \sum_{i\leqslant m} u_iu_i^t$. 
For any $i\in\nu$, define $B_i= u_iu_i^t$. Then we have $A=B+\sum_{i\in\nu^c}B_i$. Theorem~\ref{general-result-rank-one} 
follows by applying Theorem~\ref{general-result} to $A,B, (B_i)_{i\in\nu^c}$ with $m$ replaced by $\vert\nu^c\vert$. 
\end{proof}
 
 \vskip 0.3cm
 
 Theorem~\ref{general-result-rank-one} allows one to preserve the block $V$ while achieving the desired extraction. 
 Let us illustrate how this can be useful. For this aim, consider 
 the case where $UU^t=Id$.

\begin{coro}\label{isotropic-case}
Let $U$ be an $n\times m$ matrix such that $UU^t=Id$. Suppose that $U$ contains 
$r$ columns of norm $1$ for some $r\leqslant n$. Then there exists $\sigma\subset\{1,..,m\}$ of 
size $n$ such that $U_\sigma$ is invertible and 
$$
\Vert U_{\sigma}^{-1}\Vert_{\rm HS}^2\leqslant (m-n)\cdot (n-r) +n.
$$
\end{coro}

\begin{proof}
First note that $r=n$ implies that $m=n$ and the problem is trivial. Indeed, if $n<m$ we have $n={\rm Tr}\left(UU^t\right)= \sum_{i\leqslant m} \Vert u_j\Vert_2^2> r$, where 
$(u_j)_{j\leqslant m}$ denote the columns of $U$. This of course contradicts the fact that $r=n$.
  
Now let us consider the case $r<n$. 
Suppose that $UU^t=Id$ and that $U$ contains $r$ columns of norm $1$. Let $\nu\subset \{1,..,m\}$ 
be the set of indices of the norm one columns and denote $V=U_\nu$. Clearly, $V$ is an $n\times r$ matrix 
of rank $r$ since $r=\Vert V\Vert_{\rm HS}^2\leqslant \Vert V\Vert^2\cdot rank(V)$ and $\Vert V\Vert\leqslant 1$. 
By preserving $V$, we already have $r$ linearly independent columns and therefore we just need to complement them 
by $n-r$ linearly independent columns. 

Following the notations of Theorem~\ref{general-result-rank-one}, we have $A=Id$ and $B=VV^t$. Applying 
Theorem~\ref{general-result-rank-one}, one can find $\sigma'$ with
$$
\vert \sigma'\vert =  n- \left\lfloor{\rm Tr}\left(A^{-1}B\right)\right\rfloor= n-r,
$$
such that if $\sigma=\sigma' \cup \nu$, then $\vert \sigma\vert =n$ and 
\begin{align*}
\Vert U_{\sigma}^{-1}\Vert_{\rm HS}^2&\leqslant \left(\vert\nu^c\vert-n+1+{\rm Tr}\left(A^{-1}B\right)\right)\cdot{\rm Tr}\left(A^{-1}\right) 
-{\rm Tr}\left(A^{-2}B\right)\left(\vert\nu^c\vert-n+{\rm Tr}\left(A^{-1}B\right)\right)\\
&=\left(m-r-n+1+r\right)\cdot n
-r\cdot\left(m-r-n+r\right)\\
&=(m-n)\cdot (n-r) +n.
\end{align*}
\end{proof}

\begin{Rq}
When we don't look for columns of norm $1$ inside $U$ (i.e. the case where $r=0$), the estimate 
would be the same as in Theorem~\ref{gluskin-olevskii} while if $r\neq 0$, the estimate of 
Corollary~\ref{isotropic-case} improves the one of Theorem~\ref{gluskin-olevskii} by $r\cdot (m-n)$. 
Moreover, one can view Corollary~\ref{isotropic-case} as an extraction of a basis inside $U$ with 
a "good" estimate on the Hilbert-Schmidt of the inverse while preserving a "nice" block.
\end{Rq}

To have a better understanding of the generalization of Theorem~\ref{gluskin-olevskii} to the case of arbitrary rank matrices, 
let us take $B=0$ in Theorem~\ref{general-result} to get the following corollary:

\begin{coro}\label{arbitrary-rank}
Let $A, B_1,\ldots,B_m$ be $n\times n$ positive semidefinite matrices such that 
$A=\sum_{i\leqslant m} B_i$ is of rank $n$. Then for any $k\geqslant n$, there exists $\sigma\subset \{1,..,m\}$ 
of size $k$ such that $\sum_{i\in \sigma} B_i$ is invertible and
$$
{\rm Tr}\left(\sum_{i\in \sigma} B_i\right)^{-1}\leqslant \frac{m-n+1}{k-n+1}{\rm Tr}\left(A^{-1}\right).
$$
\end{coro}

\section{Proof of Theorem~\ref{general-result}}

The proof is based on an iteration of the following lemma:

\begin{lem}\label{iteration-lem}
Let $A_p, B, B_1,..., B_p$ be $n\times n$ positive semidefinite matrices such that 
$A_p=B+ \sum_{i\leqslant p} B_i$ is of rank $n$. Then there exists $j\in \{1,..,p\}$ such that 
$A_p-B_j$ is of rank $n$ satisfying $\left(A_p-B_j\right)^{-1}\succeq A_p^{-1}$ and 
$$
{\rm Tr}\left(A_p-B_j\right)^{-1}\leqslant {\rm Tr}\left(A_p^{-1}\right)\cdot\left[\frac{p-n+{\rm Tr}\left(A_p^{-1}B\right)+1}{p-n+{\rm Tr}\left(A_p^{-1}B\right)}\right] 
- \frac{{\rm Tr}\left(A_p^{-2}B\right)}{p-n+{\rm Tr}\left(A_p^{-1}B\right)}.
$$
\end{lem}

\begin{proof}

Our aim is to find $C$ among the $(B_i)_{i\leqslant p}$ such that $A_p-C$ is still invertible 
and then have a control on ${\rm Tr}\left(A_p-C\right)^{-1}$. We would like to use the Sherman-Morrison-Woodbury 
which states that 
\begin{equation}\label{sherman-morrison-woodbury}
\left( A_p -C\right)^{-1} = A_p^{-1} + A_p^{-1}C^{\frac{1}{2}}\left( Id - C^{\frac{1}{2}}A_p^{-1}C^{\frac{1}{2}}\right)^{-1}C^{\frac{1}{2}}A_p^{-1}.
\end{equation}
For this formula to hold, we should ensure choosing $C$ such that $ Id - C^{\frac{1}{2}}A_p^{-1}C^{\frac{1}{2}}$ is invertible. Since 
 $C^{\frac{1}{2}}A_p^{-1}C^{\frac{1}{2}}$ is a positive definite matrix, it is sufficient to have $\left \Vert C^{\frac{1}{2}}A_p^{-1}C^{\frac{1}{2}}\right\Vert <1$ 
 in order to ensure the invertibility of $ Id - C^{\frac{1}{2}}A_p^{-1}C^{\frac{1}{2}}$. Now since $\left \Vert C^{\frac{1}{2}}A_p^{-1}C^{\frac{1}{2}}\right\Vert \leqslant 
 {\rm Tr}\left(A_p^{-1}C\right) $ 
 then it would be sufficient to have $1-{\rm Tr}\left(A_p^{-1}C\right) >0$ in order to use (\ref{sherman-morrison-woodbury}).

Since ${\rm Tr}\left(A_p^{-2}C\right)$ is positive then we may search for $C$ satisfying 
\begin{equation}\label{condition1C}
{\rm Tr}\left(A_p^{-2}C\right) \leqslant \alpha \left(1-{\rm Tr}\left(A_p^{-1}C\right)\right),
\end{equation}
where $\alpha$ is a positive parameter which will be chosen later.

In order to guarantee the existence of $C$ among the $(B_i)_{i\leqslant p}$ that satisfies (\ref{condition1C}), it is sufficient 
to prove that (\ref{condition1C}) holds when taking the sum over all $(B_i)_{i\leqslant p}$. Therefore, we need to prove that 
\begin{equation}\label{condition-sum}
\sum_{i\leqslant p}{\rm Tr}\left(A_p^{-2}B_i\right) \leqslant \alpha \left(p-\sum_{i\leqslant p}{\rm Tr}\left(A_p^{-1}B_i\right)\right).
\end{equation}
Now since the trace is linear and $\sum_{i\leqslant p}B_i = A_p -B$ then (\ref{condition-sum}) is equivalent to the following 
\begin{equation}\label{alpha}
{\rm Tr}\left(A_p^{-1}\right)-{\rm Tr}\left(A_p^{-2}B\right)\leqslant \alpha \left(p-n+{\rm Tr}\left(A_p^{-1}B\right)\right).
\end{equation}
Choose $\alpha$ so that (\ref{alpha}) becomes true. For that let $\alpha= \frac{{\rm Tr}\left(A_p^{-1}\right) - {\rm Tr}\left( A_p^{-2} B\right)}{p-n+ {\rm Tr}\left(A_p^{-1}B\right)}$.

Therefore, there exists $j\leqslant p$ such that $C=B_j$ satisfies (\ref{condition1C}). Now, we may use the Sherman-Morrison-Woodbury formula 
(\ref{sherman-morrison-woodbury}) and write
\begin{align*}
\left( A_p -B_j\right)^{-1} &=  A_p^{-1} + A_p^{-1}B_j^{\frac{1}{2}}\left( Id - B_j^{\frac{1}{2}}A_p^{-1}B_j^{\frac{1}{2}}\right)^{-1}B_j^{\frac{1}{2}}A_p^{-1}\\
&\preceq A_p^{-1} + \frac{ A_p^{-1}B_jA_p^{-1}}{1-{\rm Tr}\left(A_p^{-1}B_j\right)}.
\end{align*}
Now since $B_j$ satisfies (\ref{condition1C}), then 
\begin{equation}\label{inverse-update-ineq}
A_p^{-1}\preceq \left( A_p -B_j\right)^{-1} \preceq A_p^{-1} + \alpha \frac{ A_p^{-1}B_jA_p^{-1}}{{\rm Tr}\left(A_p^{-2}B_j\right)}.
\end{equation}
Taking the trace in (\ref{inverse-update-ineq}), we get 
$$
{\rm Tr}\left( A_p -B_j\right)^{-1} \leqslant {\rm Tr}\left(A_p^{-1}\right) + \alpha.
$$
Replacing $\alpha$ with its value, we finish the proof of the lemma.
\end{proof}

\begin{proof}[Proof of Theorem~\ref{general-result}]
The proof will be based on an iteration of the previous lemma.  

We will construct the set $\sigma$ step by step, starting with $\sigma_0 =\{1,..,m\}$ and 
at each time taking away the "bad" indices. 
Denote by $A_0 =A= B+ \sum_{i\in\sigma_0} B_i$. Apply Lemma~\ref{iteration-lem} to find 
$j_0\in\sigma_0$ such that $A_1:=A_0-B_{j_0}$ is of rank $n$ satisfying $A_1^{-1}\succeq A_0^{-1}$ and 
\begin{equation}\label{trace-A1}
{\rm Tr}\left(A_1^{-1}\right)\leqslant {\rm Tr}\left(A_0^{-1}\right)\cdot\left[\frac{\vert\sigma_0\vert-n
+{\rm Tr}\left(A_0^{-1}B\right)+1}{\vert\sigma_0\vert-n+{\rm Tr}\left(A_0^{-1}B\right)}\right] 
- \frac{{\rm Tr}\left(A_0^{-2}B\right)}{\vert\sigma_0\vert-n+{\rm Tr}\left(A_0^{-1}B\right)}.
\end{equation}
Let $\sigma_1= \sigma_0 \setminus \{j_0\}$, then $A_1= B+ \sum_{i\in \sigma_1}B_i$ and $\vert \sigma_1\vert =m-1$. 
Now apply Lemma~\ref{iteration-lem} again in order to find $j_1\in\sigma_1$ such that 
$A_2:=A_1-B_{j_1}$ is of rank $n$ satisfying $A_2^{-1}\succeq A_1^{-1}$ and 
$$
{\rm Tr}\left(A_2^{-1}\right)\leqslant {\rm Tr}\left(A_1^{-1}\right)\cdot\left[\frac{\vert\sigma_1\vert-n+
{\rm Tr}\left(A_1^{-1}B\right)+1}{\vert\sigma_1\vert-n+{\rm Tr}\left(A_1^{-1}B\right)}\right] 
- \frac{{\rm Tr}\left(A_1^{-2}B\right)}{\vert\sigma_1\vert-n+{\rm Tr}\left(A_1^{-1}B\right)}.
$$
Since $A_0^{-1}\preceq A_1^{-1}$ then 
\begin{equation}\label{first-trace-A2}
{\rm Tr}\left(A_2^{-1}\right)\leqslant {\rm Tr}\left(A_1^{-1}\right)\cdot\left[\frac{\vert\sigma_1\vert-n+
{\rm Tr}\left(A_0^{-1}B\right)+1}{\vert\sigma_1\vert-n+{\rm Tr}\left(A_0^{-1}B\right)}\right] 
- \frac{{\rm Tr}\left(A_0^{-2}B\right)}{\vert\sigma_1\vert-n+{\rm Tr}\left(A_0^{-1}B\right)}.
\end{equation}
Let $\sigma_2=\sigma_1\setminus\{j_1\}$, then $A_2= B+ \sum_{i\in \sigma_2}B_i$ and $\vert \sigma_2\vert =m-2$. 
Combining (\ref{trace-A1}) and (\ref{first-trace-A2}), we have
\begin{equation}\label{trace-A2}
{\rm Tr}\left(A_2^{-1}\right)\leqslant {\rm Tr}\left(A_0^{-1}\right)\cdot\left[\frac{\vert\sigma_0\vert-n
+{\rm Tr}\left(A_0^{-1}B\right)+1}{\vert\sigma_1\vert-n+{\rm Tr}\left(A_0^{-1}B\right)}\right] 
- \frac{2{\rm Tr}\left(A_0^{-2}B\right)}{\vert\sigma_1\vert-n+{\rm Tr}\left(A_0^{-1}B\right)}.
\end{equation}
Suppose that we constructed $\sigma_p$ of size $m-p$ such that $A_p= B+ \sum_{i\in \sigma_p}B_i$ 
satisfies $A_p^{-1}\succeq A_0^{-1}$ and 
\begin{equation}\label{induction-step}
{\rm Tr}\left(A_p^{-1}\right)\leqslant {\rm Tr}\left(A_0^{-1}\right)\cdot\left[\frac{\vert\sigma_0\vert-n
+{\rm Tr}\left(A_0^{-1}B\right)+1}{\vert\sigma_{p-1}\vert-n+{\rm Tr}\left(A_0^{-1}B\right)}\right] 
- \frac{p{\rm Tr}\left(A_0^{-2}B\right)}{\vert\sigma_{p-1}\vert-n+{\rm Tr}\left(A_0^{-1}B\right)}.
\end{equation}
Applying Lemma~\ref{iteration-lem}, we find $j_p\in\sigma_p$ such that 
$A_{p+1}:=A_p-B_{j_p}$ is of rank $n$ satisfying $A_{p+1}^{-1}\succeq A_p^{-1}$ and 
\begin{equation}\label{first-trace-Ap}
{\rm Tr}\left(A_{p+1}^{-1}\right)\leqslant {\rm Tr}\left(A_p^{-1}\right)\cdot\left[\frac{\vert\sigma_p\vert-n
+{\rm Tr}\left(A_0^{-1}B\right)+1}{\vert\sigma_p\vert-n+{\rm Tr}\left(A_0^{-1}B\right)}\right] 
- \frac{{\rm Tr}\left(A_0^{-2}B\right)}{\vert\sigma_p\vert-n+{\rm Tr}\left(A_0^{-1}B\right)}.
\end{equation}
Let $\sigma_{p+1}=\sigma_p\setminus\{j_p\}$, then $A_{p+1}= B+ \sum_{i\in \sigma_{p+1}}B_i$ and $\vert \sigma_{p+1}\vert =m-p-1$. 
Combining (\ref{induction-step}) and (\ref{first-trace-Ap}), we have
\begin{equation}\label{trace-Ap}
{\rm Tr}\left(A_{p+1}^{-1}\right)\leqslant {\rm Tr}\left(A_0^{-1}\right)\cdot\left[\frac{\vert\sigma_0\vert-n
+{\rm Tr}\left(A_0^{-1}B\right)+1}{\vert\sigma_p\vert-n+{\rm Tr}\left(A_0^{-1}B\right)}\right] 
- \frac{(p+1){\rm Tr}\left(A_0^{-2}B\right)}{\vert\sigma_p\vert-n+{\rm Tr}\left(A_0^{-1}B\right)}.
\end{equation}
We can continue this procedure as long as $\vert \sigma_p\vert \geqslant n-{\rm Tr}\left(A_0^{-1}B\right)+1$. 
Therefore, we have proved by induction that for any $l\leqslant m - n + \left\lfloor {\rm Tr}\left(A_0^{-1}B\right)\right\rfloor$, 
there exists $\sigma_l$ of size $m-l$ such that 
$$
{\rm Tr}\left(A_{\sigma_l}^{-1}\right)\leqslant {\rm Tr}\left(A_0^{-1}\right)\cdot\left[\frac{m-n
+{\rm Tr}\left(A_0^{-1}B\right)+1}{m-n+1+{\rm Tr}\left(A_0^{-1}B\right)-l}\right] 
- \frac{l\cdot{\rm Tr}\left(A_0^{-2}B\right)}{m-n+1+{\rm Tr}\left(A_0^{-1}B\right)-l},
$$
where $A_{\sigma_l} = B+\sum_{i\in\sigma_l}B_i$.
\end{proof}

\vskip 0.5cm

\textbf{Acknowledgement :} The author would like to thank Christos Boutsidis for his valuable comments regarding this manuscript. The author would like to thank the Pacific Institute for Mathematical Science for financial support and the University of Alberta for the excellent working conditions.

\nocite{*}
\bibliographystyle{abbrv}
\bibliography{bibliography-basis-selection}




\section{Appendix}
We will present here an algorithmic proof of Theorem~\ref{gluskin-olevskii}. Our proof is inspired by the tools developed 
in \cite{MR2780071} but the procedure will be quite opposite. Write $UU^t=\sum_{i\leqslant m} u_iu_i^t$, where $(u_i)_{i\leqslant m}$ 
denote the columns of $U$. The aim is to construct $A_\sigma= \sum_{i\in\sigma} u_iu_i^t$ which satisfies the conclusion needed.
 In \cite{MR2780071}, the construction is done step by step starting from $A_0=0$ and studying the evolution of the eigenvalues 
 when adding a suitable rank one matrix. Our construction will also be done step by step, starting however with $A_0=A$ and 
 studying the evolution of the Hilbert-Schmidt norm of the inverse when subtracting a suitable rank one matrix. Precisely, at each step 
 we will remove from $U$ the "bad" columns until we have $n$ remaining "good" lineary independent columns.
 
 The proof is an iteration of the following Lemma:
 
 \begin{lem}\label{iteration-lem-rank1}
 Let $U$ be an $n\times m$ matrix of rank $n$. There exists 
 $\sigma\subset\{1,..,m\}$ of size $m-1$ such that $U_\sigma$ is 
 of  rank $n$ and 
 $$
 {\rm Tr}\left(U_\sigma U_\sigma^t\right)^{-1}\leqslant \frac{m-n+1}{m-n}{\rm Tr}\left(UU^t\right)^{-1}.
 $$
 \end{lem}

\begin{proof}
Denote $A=UU^t=\sum_{i\leqslant m}u_iu_i^t$, where $(u_i)_{i\leqslant m}$ are the columns of $U$. 
We are searching for vector $v$ chosen among the columns of $U$
 such that $A-vv^t$ is still invertible and has a control on the Hilbert-Schmidt norm 
of its inverse. We would like to use the Sherman-Morrison formula which states that if $v^tA^{-1}v\neq 1$ then
\begin{equation}\label{sherman-morrison}
{\rm Tr}\left(A-vv^t\right)= {\rm Tr}\left(A^{-1}\right)+ \frac{v^tA^{-2}v}{1-v^tA^{-1}v}. 
\end{equation}
For that, we will search for $v$ such that $v^tA^{-1}v <1$. 
Since $v^tA^{-2}v$ is positive, then it is sufficient to search for $v$ 
satisfying 
\begin{equation}\label{condition1v}
v^tA^{-2}v\leqslant \alpha \cdot \left(1-v^tA^{-1}v\right),
\end{equation}
where $\alpha= \frac{{\rm Tr}\left(A^{-1}\right)}{m-n}$.
To guarantee that such $v$ exists, we may prove that (\ref{condition1v})
 holds when taking the sum over all columns of $U$ 
$$
\sum_{i\leqslant m} u_i^tA^{-2}u_i\leqslant 
\alpha\cdot\left(m- \sum_{i\leqslant m}u_i^tA^{-1}u_i\right).
$$
This is equivalent to the following
$$
{\rm Tr}\left(A^{-2}\sum_{i\leqslant m} u_iu_i^t\right)\leqslant 
\alpha\cdot\left(m- {\rm Tr}\left( A^{-1}\sum_{i\leqslant m}u_iu_i^t\right)\right).
$$
Since $A=\sum_{i\leqslant m}u_iu_i^t$, it is then reduced to prove that
$$
{\rm Tr}\left(A^{-1}\right)\leqslant \alpha\cdot \left(m-n\right),
$$
which is true by the choice of $\alpha$.
Therefore we have found $j\in\{1,..,m\}$ such that $u_j$ satisfies (\ref{condition1v}). 
We may now use (\ref{sherman-morrison}) to get
$$
{\rm Tr}\left(A-u_ju_j^t\right)= {\rm Tr}\left(A^{-1}\right)+ 
\frac{u_j^tA^{-2}u_j}{1-u_j^tA^{-1}u_j}\leqslant \frac{m-n+1}{m-n}{\rm Tr}\left(A^{-1}\right).
$$
The Lemma follows by taking $\sigma =\{1,..,m\}\setminus \{j\}$.
\end{proof}

\begin{proof}[Proof of Theorem~\ref{gluskin-olevskii}]
Start with $A_0=A=UU^t$ and apply Lemma~\ref{iteration-lem-rank1} 
to find $\sigma_1$ of size $m-1$ such that $A_1=U_{\sigma_1}U_{\sigma_1}^t$ 
is of rank $n$ and satisfies 
\begin{equation}\label{eq-step1}
{\rm Tr}\left(A_1^{-1}\right)\leqslant \frac{m-n+1}{m-n}{\rm Tr}\left(A_0^{-1}\right).
\end{equation}
Now apply Lemma~\ref{iteration-lem-rank1} again with $U_{\sigma_1}$ to find 
$\sigma_2$ of size $m-2$ such that $A_2=U_{\sigma_2}U_{\sigma_2}^t$ is 
of rank $n$ and satisfies 
\begin{equation}\label{eq-step2}
{\rm Tr}\left(A_2^{-1}\right)\leqslant \frac{m-n}{m-n-1}{\rm Tr}\left(A_1^{-1}\right)\leqslant 
 \frac{m-n+1}{m-n-1}{\rm Tr}\left(A_0^{-1}\right).
\end{equation}
If we continue this procedure, after $k$ steps we find $\sigma_k$ of size $m-k$ such that  
$A_k=U_{\sigma_k}U_{\sigma_k}^t$ is 
of rank $n$ and satisfies 
\begin{equation}\label{eq-stepk}
{\rm Tr}\left(A_k^{-1}\right)\leqslant 
 \frac{m-n+1}{m-n-k+1}{\rm Tr}\left(A_0^{-1}\right).
\end{equation}
This holds for any $k\leqslant m-n$. After $k=m-n$ steps, Theorem~\ref{gluskin-olevskii} is proved.

\end{proof}

\vskip 0.5cm

\end{document}